\documentclass[11pt]{amsart}
\usepackage{amsmath,amssymb,bm}
\usepackage{verbatim,a4wide} 
\usepackage{graphicx}
\usepackage[utf8]{inputenc}
\usepackage{amsmath,amsfonts,amssymb,amscd,amsthm,txfonts,hyperref}

\newtheorem{theorem}{Theorem}[section]
\newtheorem{proposition}[theorem]{Proposition}
\newtheorem{lemma}[theorem]{Lemma}

\newtheorem{remark}{Remark}[section]
\newtheorem{theo}{Theorem}

\theoremstyle{definition}

\newcommand{\R}{\mathbb{R}}

\newcommand{\Z}{\mathbb{Z}}
\newcommand{\T}{\mathbb{T}}

\newcommand{\reff}[1]{(\ref{#1})}

\renewcommand{\Re}{\mbox{Re}}

\newcommand{\D}{\varmathbb D}

\begin{document}

%%%%%%%%%%%%%%%%%%%%
\title
[Random dispersive waves]
{On the propagation of weakly nonlinear random dispersive waves}
\author{Anne-Sophie de Suzzoni and Nikolay Tzvetkov}
\address{Universit\'e de Cergy-Pontoise,  Cergy-Pontoise, F-95000,UMR 8088 du CNRS }
\thanks{The authors are supported the ERC grant Dispeq}
\begin{abstract}
We study several basic dispersive models with random periodic initial data such that the different Fourier modes are independent random variables.
Motivated by the vast Physics literature on related topics, we then study how much the Fourier modes of the solution at later times remain decorrelated.
Our results are sensitive to the resonances associated with the dispersive relation and to the particular choice of the initial data.  
\end{abstract}
\maketitle
%%%%%%%%%%%%%%%%%%%%%%%%%%%%%%%%%%%%
%%%%%%%%%%%%%%%%%%%%%%%%%%%%%%%%%%%%%
\section{Introduction}
In this paper we study several basic dispersive models with random periodic initial data such that the different Fourier modes are independent random variables.
Motivated by the vast Physics literature on related topics (see e.g. \cite{Z}), we then study how much the Fourier modes of the solution at later times remain decorrelated, and how much the mean values of the amplitudes to the square of the Fourier modes vary with time.
Our results are sensitive to the resonances associated with the dispersive relation and to the particular choice of the initial data.   

All the models we will be interested in can be injected in the following general framework. Consider the equation
\begin{equation}\label{eqgen} 
( \partial_t  +L ) u + \varepsilon J (u^2) = 0, 
\end{equation}
posed on the torus $\T^d$ of dimension $d$ with an initial datum being a random variable that shall be described later. 
In \eqref{eqgen}, $\varepsilon\ll 1$ since we want to investigate about the effect of a weak non linearity 
over the behaviour of the statistics related to the random initial datum.
We suppose that $u$ is real valued and $L$ and $J$ are linear maps which are defined as Fourier multipliers by
$$
\widehat{Lu}(n)=-i\omega(n)\, \hat{u}(n),\quad
\widehat{Ju}(n)=i\varphi(n)\, \hat{u}(n),\quad \forall\,n\in\Z^d,
$$
where $\widehat{\cdot}$ denotes the Fourier transform on $\T^d$ and $\omega,\varphi :\Z^d\mapsto \R$ are supposed to be such that 
\begin{equation}\label{Hmean}
\omega(0,n')=\varphi(0,n')=0,\quad \forall\, n'\in\Z^{d-1},
\end{equation}
with the natural convention in the case $d=1$. 
We suppose that the variable on $\T^d$ is given by $x=(x_1,\hdots ,x_d)$.
Then, thanks to the assumption \eqref{Hmean}, we obtain that we can consider solutions of \eqref{eqgen} such that 
$\int_{\T} u(t,x) dx_1 = 0$.
We also suppose that $\omega,\varphi$ are odd functions. 
Observe that under the last assumption $L$ and $J$ send real valued functions to real valued functions. 
Set 
$$
\D^d = \lbrace n = (n_1,\hdots ,n_d) \in \Z^d \; |\; n_1 \neq 0\rbrace.
$$
For $s\in \R$, we introduce the Sobolev spaces $H^s$ of real functions having zero $x_1$ mean value : 
$$
H^s = \lbrace u(x) = \sum_{n\in \Z^d} e^{in\cdot x}u_n\; |\;u_n=\overline{u_{-n}} ,\quad \int_{\T} u(x) dx_1 = 0,\quad
 \sum_{n\in \D^d} |n|^{2s} |u_n|^2 < \infty\rbrace
$$
where $|n| = \sum_{j} |n_j|$. In this work we shall always make use of these Sobolev spaces $H^s$ since they are the ones adapted to our models.
In all our examples the equation \eqref{eqgen} is globally well-posed in some $H^s$ and thus there will be no difficulty caused by the problem of the existence of the dynamics.

Let us describe the dispersive models which can be written under the form  \reff{eqgen} we will be interested in.
They all appear in the modeling of long, small amplitude dispersive waves with a possible weak transverse perturbation.
The first example is the KdV equation
$$
\partial_t u+\partial_x^3u+\partial_{x}(u^2)=0
$$
which corresponds to \eqref{eqgen} in the case $d=1$ with $\omega(n)=n^3$ and $\varphi(n)=n$ (with the convention  $x=x_1$ and $n=n_1$ is the case $d=1$).
The KdV is globally well posed in $H^s$, $s\geq -1$ (see \cite{KT}, for earlier results we refer to \cite{Bo,CKSTT,KPV}).

A second example again in the case $d=1$ is an alternative of the KdV model, derived by Benjamin-Bona-Mahony (BBM equation) which can be written as
$$
\partial_t u+\partial_x u-\partial_t\partial_x^2u+\partial_{x}(u^2)=0.
$$
The BBM equation corresponds to \eqref{eqgen} with $-\omega(n)=\varphi(n)=n/(1+n^2)$. The BBM equation is globally well-posed in $H^s$, $s\geq 0$ 
(see \cite{Bona,D}).

Our two dimensional models will be the famous Kadomtsev-Petviashvili  (KP) equations. In fact there are two models according to the impact of the surface tension.
The first one is the KP-II equation which corresponds to a weak surface tension and can be written as  
$$
\partial_t u+\partial_{x_1}^3u+\partial_{x_1}^{-1}\partial_{x_2}^2 u+\partial_{x_1}(u^2)=0.
$$
The KP-II equation corresponds to \eqref{eqgen} in the case $d=2$ with 
$\omega(n_1,n_2)=n_1^3-n_2^2/n_1$ if $n_1\neq 0$, $\omega(0,n_2)=0$ and $\varphi(n_1,n_2)=n_1$.
The KP-II equation is globally well-posed in $H^s$, $s\geq 0$ (see \cite{Bo-KP}).

Finally, the KP-I equation 
$$
\partial_t u+\partial_{x_1}^3u-\partial_{x_1}^{-1}\partial_{x_2}^2 u+\partial_{x_1}(u^2)=0.
$$
corresponds to \eqref{eqgen} with 
$\omega(n_1,n_2)=n_1^3+n_2^2/n_1$ if $n_1\neq 0$, $\omega(0,n_2)=0$ and $\varphi(n_1,n_2)=n_1$.
The KP-I equation is globally well-posed if the data is in $H^s$, $s\geq 2$ (see \cite{IK} and also \cite{IKT}). 

Next, we describe the random initial data we shall deal with. With $\D^d_+ = \lbrace n \in \D^d \; |\; n_1>0 \rbrace$, 
let $(g_n)_{n\in \D^d_+}$ be a sequence of independent identically distributed complex random variables such that 
$$
E(g_n)=0,\quad  E(|g_n|^2) = 1
$$
and such that there exist to positive constants $c$ and $C$ such that for all $\gamma \in \R$,
\begin{equation}\label{sub-gauss}
E(e^{\gamma \textrm{Re} (g_n)})\leq C\,e^{c\gamma^2},\quad E(e^{\gamma \textrm{Im} (g_n)})\leq C\, e^{c\gamma^2},
\end{equation}
where $E$ is the expectation. We also suppose that the distribution of $g_n$ is invariant under the multiplication by $e^{i\theta}$ with $12\theta  \neq 0 [2\pi]$.
Note that under these assumptions, $E(g_n^2)$ is equal to $0$. Further consequences of this property will be used in the sequel.
\begin{remark}
{\rm 
A typical example of random variables satisfying our assumptions are the (complex) Gaussian random variables, i.e.
$g_n=\frac{1}{\sqrt{2}}(h_n+il_n)$, with $h_n,l_n\in {\mathcal N}(0,1)$. 
Another example coming from the Physics literature is what is known as random phase approximation, that is, $g_n$ is written $g_n = \chi_n A_n$, where $\chi_n$ is uniformly distributed on $S^1$ and $A_n$ is a non-negative random variable independent from $\chi_n$, and $E(A_n^2)=1$.
In all these examples the symmetry assumption on $g_n$ holds with any angle $\theta\neq 0$.
In order to ensure \eqref{sub-gauss}, we can suppose that the distribution $\mu$ of $A_n$ satisfies 
$$
\int_{0}^{\infty}e^{\gamma r}d\mu(r)\leq C e^{c\gamma^2}\,.
$$
For instance, the last property holds true if $\mu$ is compactly supported.
}
\end{remark}
%%%%%
Next, for $n\in \D^d_+$ set $g_{-n} = \overline{g_n}$. 
Let $\lambda=(\lambda_n)_{n\in \D^d_+}$ be a sequence of complex numbers such that
\begin{equation}\label{Hs}
\sum_{n\in \D^d_+} |n|^{2s} |\lambda_n|^2 < \infty
\end{equation}
for some $s$ depending on $L$ and $J$ such that the equation \eqref{eqgen} 
is globally well-posed in $H^s$. Set for all $n\in \D^d_+$, $\lambda_{-n} = \overline{\lambda_n}$.
Set 
\begin{equation}\label{data}
u_0(x) = \sum_{n\in \D^d} g_n \lambda_n e^{in\cdot x}
\end{equation}
Thanks to our assumption on $(\lambda_n)_{n\in \D^d_+}$, we have that $u_0\in H^s$ almost surely. Moreover, it is real valued. 
Let $u(\varepsilon;t, x)$ be the solution of 
$$
\left \lbrace{\begin{tabular}{ll}
$\left(\partial_t + L\right) u +\varepsilon J (u^2) =0$\\
$u_{|t=0} = u_0 .$\end{tabular}} 
\right. 
$$
Consider the expansion of  $u(\varepsilon ;t, x)$ as a Fourier series,
$$
u(\varepsilon;t, x)=\sum_{n\in \D^d} u_{n}(\varepsilon ; t) e^{in\cdot x}\,.
$$
Set $S(t) = e^{-tL}$. Then clearly
$u(0;t,x) = S(t)u_0$ and
$$
u_{n}(0;t) =e^{i\omega(n) t}g_n \lambda_n\,.
$$
In particular, thanks to our assumption on the random variables $g_n$, 
\begin{equation}\label{indep}
E(\overline{u_m(0;t)}u_n(0;t)) = \delta_n^m |\lambda_n |^2 \, ,\quad \forall\, t\in\R .
\end{equation}
Our aim is to expand the quantity $E(\overline{u_m(\varepsilon;t)}u_n(\varepsilon;t))$ in $\varepsilon$ and see how much \eqref{indep} survives in the
nonlinear setting. 

In order to state our result, we introduce the following notations.  We set
$$\Delta_{n}^{k,l} = \omega(k) + \omega(l) -\omega(n)$$
which corresponds to the pulsation associated to the three waves interaction $k+l \rightarrow n$ when $k+l = n$.
Next we set :
$$
F_n^{k,l}(t) = \int_{0}^t e^{i\Delta_n^{k,l}\tau}d\tau \; .
$$
Here is our main result.
%%%
\begin{theo}\label{result}
Consider \eqref{eqgen}, in the cases of the KP-II, BBM and the KP-I equations,  
with initial data given by \eqref{data} of typical Sobolev regularity $H^s$, $s>3/8$ for BBM, $s>2$ for $KP-II$ and $s>3$ for $KP-I$.
Then 
\begin{equation}\label{mea}
E(\overline{u_m(\varepsilon;t)}u_n(\varepsilon;t)) = 
\delta_n^m |\lambda_n|^2 + \delta_n^m\varepsilon^2 G_n(\lambda,t) + \varepsilon^3 R(\varepsilon;t,m,n)
\end{equation}
where $G_n(\lambda,t)$ is given by $G_n(\lambda,0) = 0$ and
\begin{eqnarray*}
\partial_t G_n(\lambda,t) & = & 4\varphi(n)  \sum_{k+l=n}  \Re (-F_n^{k,l}(-t)) 
\Big( \varphi(n) |\lambda_k |^2 |\lambda_l|^2 -\varphi(k) |\lambda_n|^2 |\lambda_l |^2 -\varphi(l) |\lambda_n|^2 |\lambda_k |^2 \Big)
\\
&  &
+(E(|g_n|^4)-2)
\Big(
2\delta_{n}^{2q} \Re (-F_n^{q,q}(-t))\varphi^2(n)|\lambda_q|^4- 4 \Re (-F_n^{2n,-n}(-t))\varphi(2n)\varphi(n)|\lambda_n|^4
\Big)
\end{eqnarray*}
and besides $G_n(\lambda, t)$ and $R(\varepsilon;t,m,n)$ satisfies the following estimates.
There exists $C>0$ such that for every $ \varepsilon \in (0,1]$, every $|t| \leq \frac{1}{C\varepsilon}$, every $m,n$,
$$
|G_n(\lambda,t)|\leq C  |n|^{-\beta(s)},\qquad 
|R(\varepsilon;t,m,n)|\leq C 
\min(|n|, |m|)^{-1}\,
%\frac{1}{|n|\; |m|} 
|t|(1+|t|) 
$$
in the case of the BBM equation, with $\beta (s) = 2+2s$ if $s\geq 1$ and $\beta(s) = 4s$ otherwise,
$$
|G_n(\lambda,t)|\leq C |n|^{-2s},\qquad 
|R(\varepsilon;t,m,n)|\leq C \max (|n|,|m|) |t|(1+|t|) 
$$
in the case of the KP-II equation,  and
$$
|G_n(\lambda,t)|\leq Ct^2 |n|^{2-2s},\qquad 
|R(\varepsilon;t,m,n)|\leq C\max (|n|,|m|)|t|^3
$$
in the case of the KP-I equation.
\end{theo}
%%%%%%
Notice that, in the case of BBM and KP-II, only for $|t|=o(\varepsilon^{-1/2})$, the third term in \eqref{mea} in negligible with respect to the second.

It is possible that the estimates on the remainder $R$ could be improved.

It is remarkable that in the case of the BBM equation, if $\lambda_k=\frac{1}{\sqrt{1+k^2}}$ and $E(|g_n|^4)=2$ then  $G_n(\lambda,t)=0$. 
This goes with the fact, due to the first author \cite{dS}, that the measure (on $H^{1/2-}$) induced by
$$\sum_n \frac{g_n}{\sqrt{1+n^2}} e^{in x}, $$
where $g_n$ are Gaussian variables, is invariant by the flow of BBM. Indeed, this measure is a renormalization of the formal measure
$$e^{- \|u\|_{H^1}^2}du$$
and the $H^1$ norm of the solution of BBM is conserved by the evolution. 
In this particular case for $\lambda_n$ and $g_n$ the terms of higher order should also vanish as shows the next proposition.
\begin{proposition}\label{miracle}
Consider the BBM equation.
Let $\lambda_k=\frac{1}{\sqrt{1+k^2}}$  and let $g_n=\frac{1}{\sqrt{2}}(h_n+il_n)$, with $h_n,l_n\in {\mathcal N}(0,1)$. 
Then with the notations of Theorem~\ref{result},
$$
G_n(\lambda,t) =  R(\varepsilon;t,m,n)=0,\qquad \forall\, (m,n),\,\, \forall\, t,\,\, \forall\, \varepsilon.
$$
\end{proposition}
However, in the proof of Theorem~\ref{result}, 
the computation of $G_n$ depends only on $E(|g_n|^2)$ and $E(|g_n|^4)$, 
which gives a larger framework for almost remaining decorrelated initial data.
The assumptions on the random variables that they have large Gaussian deviation estimates is imposed in order get the analytic bounds on
$|R(\varepsilon;t,m,n)|$. 
%%%%
\begin{remark}{\rm The idea underlying the computation of $G_n$ comes from the theory of wave turbulence and the notion of statistical equilibrium. Indeed, as stochastic laws invariant through the flow of one conservative Hamiltonian  PDE tend to be quite rare, 
and to broaden our views on the topic, statistical equilibrium is defined as the next best thing, that is, a law whose moments of order $2$, i.e. the $E(|u_n|^2)$ are unchanged by the evolution in time. With $F_n((y_m)_{m\in \D^d}, t) = \partial_t G_n((\sqrt{y_m})_m,t)$, if we replace $|\lambda_m|^2$ by $E(|u_m|^2) + O(\varepsilon^2)$ in the following expression of $\partial_t E(|u_n|^2)$, we formally get that : 
$$
\partial_t E(|u_n|^2) = \varepsilon^2 \partial_t G_n(\lambda, t) + O(\epsilon^3) = \varepsilon^2 F_n((E(|u_m|^2))_m, t) + O(\varepsilon^3)\; .
$$
Hence, by neglecting the remainder because of its order in $\varepsilon$, we have a closed equation on the $E(|u_m|^2)$ detecting statistical equilibrium : 
$$
\forall t,\forall n, \; \partial_t E(|u_n|^2) = F_n((E(|u_m|^2))_m,t) = 0\; .
$$
When one takes the weak limit of $F_n$ when $t$ goes to $\infty$, only the resonance terms remain. In this sense, 
$$\forall n\; ,\; \lim_{t\rightarrow \infty} F_n((E(|u_m|^2))_m,t) = 0$$
is the kinetic equation corresponding to statistical equilibrium or KZ spectra in the wave turbulence theory. 
Namely, if for instance  $E(|g_n|^4)=2$ then the equation can be written as
$$
\forall n\; ,\; 
\sum_{k+l=n,\Delta_n^{k,l}=0} \left(\varphi(n) (E(|u_k|^2))  (E(|u_l|^2)) - \varphi(k) (E(|u_n|^2))  (E(|u_l|^2)) -\varphi(l) (E(|u_k|^2))  (E(|u_n|^2)) \right) = 0 \; .
$$
One can see that it depends only on the $E(|u_n|^2)$, meaning that the solution is invariant through dephasing on each wave length.
We did not make a serious effort in either finding solutions of this equation or of the more general 
\begin{equation}\label{est}
\forall n\; \forall t \; ,\; G_n(\lambda,t)=0
\end{equation}
 but we believe that the solutions of the first one, would they exist, would be consistent with the KZ spectra and the actual status of the wave turbulence theory. 
These solutions would act as a statistics into which the difference between $E(|u_n(\varepsilon; t)|^2)$ and its initial value would be negligible.
} \end{remark}
 \begin{remark}{\rm
 As a matter of a simple observation, inspired by the discussion on the BBM equation, we have that
 if
 $E(|g_n|^4)=2$ 
 then
 $
 |\lambda_{k}|^2=\varphi(k)/k_1,
 $
 is a solution of \eqref{est}. 
 In such a situation the quantity $E(|u_n(\varepsilon;t)|^2)$ is the same as its initial value at $t=0$ up to a correction of order $\varepsilon^3$, at least for times of order 
 $1$. }
 \end{remark}

\begin{remark}{\rm It also seems that one should be able to get similar results for 
$$
 E(\overline{u_{m_1} ...u_{m_k}}  u_{n_1} ...u_{n_k}) ) \; ,
$$
with $k>1$, thus approaching the more general law of the solution instead of only the covariances between the amplitudes of the Fourier modes.}
\end{remark}
Let us observe that the results of Theorem~\ref{result} for the KP equation equations also apply for the KdV equation, by considering data independent 
of the transverse variable $x_2$. The result for the BBM and KP-II equations is stronger compared to the result for the KP-I equation thanks to the absence of resonance interactions. 

The regularity assumptions of Theorem~\ref{result} are more restrictive the ones required by the well-posedness results quoted above. 
It is a natural open question whether in  Theorem~\ref{result} one can cover the weaker regularity assumptions of the well-posedness results. We reckon that some new phenomenons may occur at low regularities.

Let us now explain the main ingredients of the proof of Theorem~\ref{result}. The first step is to get deterministic bounds on the first two Picard iterations. 
Similarly to the Cauchy problem analysis, the presence of resonances plays an important role in the control of the second iteration.
We use some algebraic cancellations of the average between the first and the second iterations. Similar computations appear in the
Physics literature. The main novelty in our work is the control on the remainder (once one singles out the first two iterations). Here we use an energy
method based on a conservation law together with the exponential integrability of the first two iterations for times of order $\lesssim \varepsilon^{-1}$.   

The remaining part of this paper is organized as follows. In the next section we prove Proposition~\ref{miracle}.
In the subsequent sections, we prove Theorem~\ref{result}.
%%%%%%%%%%%%%%%%%%%%%%%%%%%%%%%%%%%%%%%%%%%%%%%%%%%%%%%%%%%%%%%%%%%%%%
%%%%%%%%%%%%%%%%%%%%%%%%%%%%%%%%%%%%%%%%%%%%%%%%%%%%%%%%%%%%%%%%%%%%%%%
\section{Proof of Proposition~\ref{miracle}}
Denote by $\mu$ the measure on $H^{1/2-}$ induced by the map
$$
\omega\longmapsto \sum_n \frac{g_n(\omega)}{\sqrt{1+n^2}} e^{in x}\equiv u^\omega \,.
$$
Denote by $\Phi(t)$ the global flow of the BBM on $L^2$ defined in \cite{D}.
Thanks to \cite{dS}, 
\begin{equation}\label{invariance}
\int_{H^{1/2-}}F(u)d\mu(u)=\int_{H^{1/2-}}F(\Phi(t)(u))d\mu(u),\quad
\forall\, t\in\R,\,\, \forall\, F\in L^1(d\mu).
\end{equation}
Denote by $\Pi_n$ the projection to the $n$'th Fourier mode.
Then
\begin{eqnarray*}
E(\overline{u_m(\varepsilon;t)}u_n(\varepsilon;t)) & = &
\int_{\Omega}\overline{\Pi_m \Phi(t)(u^\omega)}\Pi_n \Phi(t)(u^\omega)dp(\omega)
\\
& = &
\int_{H^{1/2-}}\overline{\Pi_m \Phi(t)(u)}\Pi_n \Phi(t)(u)d\mu(u)\,.
\end{eqnarray*}
Using \eqref{invariance} with $F(u)=\overline{\Pi_m (u)}\Pi_n (u)$, we get
\begin{eqnarray*}
E(\overline{u_m(\varepsilon;t)}u_n(\varepsilon;t)) & = &
\int_{H^{1/2-}}\overline{\Pi_m (u)}\Pi_n (u)d\mu(u)
\\
& = &
\int_{\Omega}\overline{\Pi_m (u^\omega)}\Pi_n (u^\omega)dp(\omega)
\\
& = & 
\delta_n^m |\lambda_n|^2\,.
\end{eqnarray*}
This completes the proof of Proposition~\ref{miracle}.
%%%%%%%%%%%%%%%%%%%%%%%%%%%%%%%%%%%%%%%%%%%
\section{Deterministic estimates for the expansion at order 2 of the solutions}
In this section $u_0$ is a deterministic $H^s$ function.
Consider \eqref{eqgen} with data $u_0$.
We suppose that
$$
u_0(x) = \sum_{n\in \D^d} a_n e^{in\cdot x},\quad  a_{-n} = \overline{a_n}\,.
$$
Let us expand the solution of \eqref{eqgen} with data $u_0$ at order $2$ in $\varepsilon$.
For simplification in the computations, let 
$$v(\varepsilon;t,x) \equiv S(-t) u(\varepsilon, t ,x) = \sum_{n\in \D^d} v_n(\varepsilon,t) e^{in\cdot x}$$
such that $v$ satisfies : 
$$
\partial_t v = -\varepsilon S(-t)J \big( (S(t)  v)^2\big)
$$
with initial datum $u_0$. Write then
$$
v(\varepsilon;t,x) = u_0(x) + \varepsilon b(t,x) + \varepsilon^2 c(\varepsilon;t,x) 
$$
and
$$
v_n(\varepsilon;t) = a_n + \varepsilon b_n(t) + \varepsilon^2 c_n(\varepsilon;t)
$$
with 
$$
b(t,x) = -\int_{0}^t S(-\tau)J \big( (S(\tau) u_0 )^2\big)d\tau \; .
$$
Then
$$
b_n(t) = - i\varphi(n) \sum_{k+l=n} a_k a_l F_{n}^{k,l}(t)\,.
$$
The next no resonance lemma plays a key role in out analysis.
\begin{lemma}
Let $k+l=n$. Then KP-II and BBM present no resonances, that is for KP-II
$$
|\Delta_n^{k,l}| \geq 3 |n_1k_1l_1 |\neq 0
$$
and for BBM
$$
\Delta_n^{k,l} = \frac{nkl(k^2+l^2+kl+3)}{(1+n^2)(1+k^2)(1+l^2)} \neq 0\; .
$$
\end{lemma}
The proof is a straightforward computation.
The consequence of this lemma is that the norm of $b$ can be bounded independently from $t$ for KP-II and BBM.

\begin{lemma}\label{second_iteration}
Suppose that for KP-II, the initial datum $u_0 $ belongs to $H^s$ with $s> \frac{1}{2}$, then
there exists $C$ independent of $t$ and $u_0$ such that
$$\|b(t)\|_{H^s} \leq C \|u_0\|_{H^s}^2 \; .$$
For BBM, suppose that the initial datum is in $H^s$, with $s > \frac{1}{4}$ and let $\sigma$ such that $0 \leq \sigma <2s - \frac{1}{2}$, then if $s\leq 1$,
$$
\|b(t)\|_{H^\sigma} \leq C \|u_0\|_{H^s}^2\; ,
$$
and if $s > \frac{1}{2}$,
$$
\|b(t)\|_{H^s} \leq C \|u_0\|_{H^s}^2\; .
$$
\end{lemma}
\begin{proof} We use the form of $b_n$ to give the bound
$$|n|^{2s} |b_n|^2 \leq |n|^{2s} |\varphi(n)|^{2}\sum_{k,l}| a_k a_{n-k}\overline{a_l}\overline{a_{n-l}}| |F_n^{k,n-k}(t)|\; |F_n^{l,n-l}(t)| \; .$$

Then, as $s\geq 0$, $|n|^s \leq C_s (|k|^s+ |n-k|^s)$, and using the facts that the sum is symmetric in $k$, $n-k$ and $l$, $n-l$ and that there is no resonances then 
$|F_{n}^{k,n-k}(t)|\leq |\frac{2}{\Delta_n^{k,n-k}}|$,
$$
|n|^{2s}|b_n|^2  \leq C_s |\varphi(n) | \sum_{k,l} |a_k|\; |n-k|^s |a_{n-k}| |a_l|\;|n-l|^s|a_{n-l}|\frac{1}{|\Delta_n^{k,n-k}\Delta_n^{l,n-l}|}\; .
$$
For KP-II, $\frac{|\varphi(n)|}{|\Delta_n^{k,n-k}|}\leq \frac{1}{3|k_1|}$, thus by summing over $n$ and using a Cauchy-Schwartz inequality : 
$$
\|b\|_{H^s}^2\leq \sum_{k,l}\frac{|a_k|}{|k_1|}\frac{|a_l|}{|l_1|}\sum_n |n-k|^s |a_{n-k}|\; |n-l|^s |a_{n-l}|
$$
$$
\|b\|_{H^s} \leq C_s\|u_0\|_{H^s} \Big|\sum_k \frac{|a_k|}{|k_1|} \Big|
\leq C_s \|u_0\|_{H^s}^2\sqrt{\sum_{k\in \D^d} \frac{1}{|k_1|^2\; |k|^{2s}}}
$$
and as $d=2$, the series converges as long as $s> \frac{1}{2}$.

In the case of BBM, we have : 
$$
\Delta_n^{k,l} = \frac{nkl(k^2+l^2+kl+3)}{(1+n^2)(1+k^2)(1+l^2)}\; .
$$
As for $s \in [-1,1]$, we have : 
$$
|k|^{s+1}|l|^{1-s} \leq k^2+l^2 \leq 2(k^2+l^2+kl) \; ,
$$
we conclude that
$$
|\Delta_n^{k,l} | \geq C \frac{|n|}{1+n^2} \frac{|k|^{s+2}}{1+k^2}\frac{|l|^{2-s}}{1+l^2} \geq C |\varphi(n) | \frac{|k|^s}{|l|^s}\; .
$$
Hence, 
$$
\frac{|\varphi(n)|}{|\Delta_n^{k,l}|} \leq C |l|^s|k|^{-s}\; .
$$
Let us now bound the $H^\sigma$ norm of $b$ in terms of the $H^s$ norm of $u_0$.
Since
$$
b_n = -i\varphi(n) \sum_{k+l = n} a_k a_l F_n^{k,l}(t)
$$
we have that
$$
|b_n| \leq |\varphi(n) |\sum_{k+l=n} |a_k|\; |a_l| \frac{2}{|\Delta_n^{k,l}|}
$$
Using that for $\sigma \geq 0$, $|n|^\sigma \leq C_\sigma (|k|^\sigma +|l|^\sigma)$ and by symmetry of the sum over $k$ and $l$ : 
$$
|n|^\sigma |b_n| \leq C \sum_{k+l=n }|k|^\sigma |a_k|\; |a_l|\; \frac{|\varphi(n)|}{|\Delta_n^{k,l}|}\; .
$$
We then use the bound on $ \frac{|\varphi(n)|}{|\Delta_n^{k,l}|}$  to write
$$
|n|^\sigma |b_n| \leq C \sum_{k+l=n }|k|^{\sigma - s} |a_k| \; |l|^s |a_l|\; .
$$
and therefore
$$
\|b\|_{H^\sigma}^2 \leq C \sum_n \sum_{k,j} |k|^{\sigma - s} |a_k| \; |n-k|^s |a_{n-k}||j|^{\sigma - s} |a_j| \; |n-j|^s |a_{n-j}| \; .
$$
\\
By reversing the order of the sums and using a Cauchy-Schwartz inequality on the sum over $n$ : 
$$\|b\|_{H^\sigma} \leq C ||u_0||_{H^s}\Big( \sum_k |k|^{\sigma -s }|a_k| \Big) $$
and since 
$$
\sum_k |k|^{\sigma -s }|a_k| \leq \left( \sum_k |k|^{2\sigma-4s} \right) ^{1/2}\|u_0\|_{H^s}
$$
and the series converges if $s> \frac{1}{4}+\frac{\sigma}{2}$, we get : 
$$
\|b\|_{H^\sigma} \leq C \|u_0\|^2_{H^s} \; .
$$
For $s\geq 1$ we simply use $|\Delta_n^{k,l}|\geq |\varphi(n)|$ and an argument similar to the one for KP-II yields the claimed bound.
This completes the proof of Lemma~\ref{second_iteration}.
\end{proof}
\begin{remark}
The arguments we presented here for the KP-II equation relax the assumption $s>1$ to $s>1/2$ in \cite{tz}.
\end{remark}
For the KP-I equation, there are resonances and hence a much weaker statement holds. 
\begin{lemma}
For KP-I, it appears that for $s> 1$
$$
\|b\|_{H^{s-1}} \leq C|t| \|u_0\|_{H^s}^2 \; .
$$
\end{lemma}
\begin{proof}Use the expression of $b$ to get the bound
$$
\|b(t)\|_{H^{s-1}} \leq 
\int_{0}^t \|J \left( (S(\tau)u_0)^2  \right) d\tau \|_{H^{s-1}} \leq 
C_s\int_{0}^t \|S(\tau)u_0\|_{L^\infty}\|S(\tau)u_0\|_{H^s} \leq C_s |t|\; \|u_0\|^2_{H^s}
$$
since as $d=2$ and $s>1$, the Sobolev embedding $H^s \subset L^\infty$ holds.
\end{proof}
%%%%%%%
\begin{lemma}The map $c$ satisfies : 
$$
\partial_t c  = -S(-t) J \left( 2 S(t) u_0 S(t) b + \varepsilon ((S(t)b)^2 + 2S(t)u_0 S(t) c) ) + \varepsilon^2 2S(t) b S(t) c + \varepsilon^3 (S(t)c)^2 \right). 
$$
\end{lemma}
\begin{proof}
It comes from a combination of the equations satisfied by $v$ and $b$.
\end{proof}
We now would like to prove that $c$ is of order $0$ in $\varepsilon$ but that its order in time depends on the cases, whether the equation displays resonances or not.
\begin{lemma}\label{energy}
For KP equations, one can bound the $L^2$ norm of $c$. In the case of KP-I (with resonances), it comes if $s>3$
$$
\|c(t)\|_{L^2}\leq 
C \left( t^2 \|u_0\|_{H^s}^3 + \varepsilon |t|^3 \|u_0\|_{H^s}^4 \right) e^{c\varepsilon\, |t|\,\|u_0\|_{H^s}(1+\varepsilon |t| 
\|u_0\|_{H^s})}\; .
$$
And for KP-II, it comes if $s>2$
$$
\|c(t)\|_{L^2}\leq C |t| \left( \|u_0\|_{H^s}^3 + \varepsilon \|u_0\|_{H^s}^4 \right) e^{c\varepsilon \,|t|\, \|u_0\|_{H^s}(1+\varepsilon \|u_0\|_{H^s})}\; .
$$
For BBM, the relevant quantity is the $H^1$ norm of $c$, it comes if $s>3/8$ : 
$$
\|c(t)\|_{H^1} \leq  C|t| (\|u_0\|_{H^s}^3 +\varepsilon \|u_0\|_{H^s}^4 ) e^{c \varepsilon\, |t|\, \|u_0\|_{H^s}(1+ \varepsilon \|u_0\|_{H^s})} 
\; .
$$
\end{lemma}
%%%%%
\begin{proof} Calling $E(t) = \frac{1}{2}||c(t)||_{L^2}$ for KP and $E(t) = \frac{1}{2}||c(t)||_{H^1}$ for BBM,
$$E(t) \partial_t E(t) = I + II+ III$$
with
$$
I = -\int c S(-t)\partial_{x_1} \left( 2S(t) u_0 S(t) b + \varepsilon (S(t)b)^2\right) \; ,
$$
$$ II = - \int c S(-t) \partial_{x_1} \left( \varepsilon 2S(t) u_0 S(t) c + \varepsilon^2 2 S(t)b S(t) c\right)$$
and
$$
III = -\varepsilon^3 \int c S(-t) \partial_{x_1} (S(t)c)^2 = -\frac{\varepsilon^3}{3} \int \partial_{x_1} (S(t) c)^3 = 0 \; .
$$
For KP equations, it appears that
$$
I(t) \leq C \|S(t) c \|_{L^2} \|\partial_{x_1}\left( 2S(t) u_0 S(t) b + \varepsilon (S(t)b)^2 \right)\|_{L^2}
$$
and therefore
$$
I(t) \leq C E(t)
\left( \|\partial_{x_1} S(t) u_0\|_{L^2} \|S(t) b \|_{L^\infty} + \|\partial_{x_1} S(t) b\|_{L^2} \|S(t) u_0 \|_{L^\infty} + \|\partial_{x_1} S(t) b\|_{L^2} \|S(t) b \|_{L^\infty}  
\right)
$$
Using that the $H^s$ norms are invariant through the flow $S(t)$, as $s\geq 1$ in both cases, 
$$
\|\partial_{x_1} S(t) u_0\|_{L^2} \leq \|S(t) u_0\|_{H^1} = \| u_0\|_{H^1}\leq \|u_0\|_{H^s}
$$
$$
I(t) \leq C E(t) \left( \|u_0\|_{H^s}\|S(t) b\|_{L^\infty} + \|b\|_{H^1}\|S(t)u_0\|_{L^\infty} + \varepsilon \|b\|_{H^1}\|S(t) b \|_{L^\infty}\right)
$$
and using the fact that $\int f\partial_{x_1}(fg) = \frac{1}{2}\int f^2 \partial_{x_1}g$, for KP we have 
$$
II(t) =- \int (S(t)c)^2 \partial_{x_1} \varepsilon S(t)u_0 - \int (S(t)c)^2 \varepsilon^2 \partial_{x_1} S(t)b
$$
$$
II(t) \leq C \|S(t) c\|_{L^2}^2 \left( \varepsilon \|\partial_{x_1} S(t) u_0\|_{L^\infty} + \varepsilon^2 \|\partial_{x_1} S(t) b\|_{L^\infty} \right)
$$
and thus
$$
II(t) \leq C E(t)^2\left( \varepsilon \|u_0\|_{H^s} + \varepsilon^2 \|\partial_{x_1} S(t) b\|_{L^\infty} \right)\; .
$$
Then, for KP-I, use the fact that for $s>3$ ($s-2>1$), $H^{s-2}$ injects itself in $L^\infty$
$$
\|\partial_{x_1} S(t) b\|_{L^\infty} \leq C \|\partial_{x_1}S(t)b\|_{H^{s-2}} \leq C \|S(t)b\|_{H^{s-1}} \leq C|t| \|u_0\|_{H^s}^2
$$
and thus
$$
\partial_t E(t) \leq C \left( |t| \; \|u_0\|_{H^s}^3 + \varepsilon t^2 \|u_0\|_{H^s}^4\right) + C E(t) \left( \varepsilon \|u_0\|_{H^s} + \varepsilon^2 |t| \; \|u_0\|^2_{H^s}\right)\; .
$$
With a Gronwall lemma,
$$
\|c(t)\|_{L^2}\leq C 
\left( t^2 \|u_0\|_{H^s}^3 + \varepsilon |t|^3 \|u_0\|_{H^s}^4 \right) e^{\varepsilon c|t| \; \|u_0\|_{H^s} + \varepsilon^2 t^2 \|u_0\|_{H^s}^2}\; .
$$
For KP-II, use the fact that for $s>2$, $H^{s-1}\subset L^\infty$,
$$
\|\partial_{x_1} S(t) b\|_{L^\infty} \leq C\|\partial_{x_1}S(t)b\|_{H^{s-1}} \leq C \|S(t)b\|_{H^{s}} \leq C \|u_0\|_{H^s}^2
$$
and thus
$$
\partial_t E(t) \leq C \left( \|u_0\|_{H^s}^3 + \varepsilon  \|u_0\|_{H^s}^4\right) + 
C E(t) \left( \varepsilon \|u_0\|_{H^s} + \varepsilon^2  \|u_0\|^2_{H^s}\right)
$$
and therefore
$$\|c(t)\|_{L^2}\leq C 
\left( |t| \|u_0\|_{H^s}^3 + \varepsilon |t|\; \|u_0\|_{H^s}^4 \right) e^{\varepsilon c|t| \; \|u_0\|_{H^s} + \varepsilon^2 |t|\; \|u_0\|_{H^s}^2}\; .
$$
For BBM, $c$ satisfies : 
$$
2(1-\partial_x^2)\partial_t c = 
 -S(-t) \partial_x \left( 2S(t)u_0 S(t)b + \varepsilon (S(t)b)^2 + \varepsilon 2S(t)u_0 c + \varepsilon^2 2S(t)b S(t)c + \varepsilon^3 (S(t)c)^2\right) \; .
$$
Since $s> \frac{3}{8}$, we can choose $\sigma$ in $]\frac{1}{4},2s-\frac{1}{2}[$ if $s\leq 1$ and $\sigma = s $ otherwise. We have then that :
$$
\partial_t \|c(t)\|_{H^1} \leq 
C \left( \|S(t)u_0S(t)b\|_{L^2} + \varepsilon \|S(t)b\|_{L^4}^2 + \|c(t)\|_{H^1}( \varepsilon \|S(t) u_0\|_{L^2} +\varepsilon^2 \|S(t) b\|_{L^2})\right)
$$
and as $s,\sigma > 1/4$, the Sobolev embeddings $H^s \subset L^4$ and $H^\sigma \subset L^4$ hold,
$$
\|S(t)u_0S(t)b\|_{L^2} \leq \|S(t)u_0\|_{L^4}\|S(t) b \|_{L^4} \leq C \|S(t)u_0\|_{H^s} \|S(t) b\|_{H^\sigma} \leq C \|u_0\|_{H^s}^3
$$
Finally,
$$
\|c(t)\|_{H^1} \leq 
C|t| (\|u_0\|_{H^s}^3 +\varepsilon \|u_0\|_{H^s}^4 ) e^{|t| (\varepsilon \|u_0\|_{H^s} + \varepsilon ^2 \|u_0\|_{H^s}^2)} \; .
$$ 
This completes the proof of Lemma~\ref{energy}.
\end{proof}
\begin{remark}\label{rem}{\rm
One may also establish estimates for higher order derivatives of $c$ by the classical energy method.
This method does not give the cancellation of the term $III$ above and thus the restriction of the time for which the estimate holds depends on $u_0$ and thus on the probability event of which $u_0$ is a representation. 
In particular, it is not clear to us how to exploit in general such an estimate in the context of the study of the decorrelation  of the  Fourier modes of the solution.
Nevertheless, by using random variables $g_n$ with values in a compact set, we should be able to use the energy method with a time of validity that does not depend on the probability event. For instance, one can use $g_n = \chi_n A_n$ where $\chi_n$ is uniformly distributed on $S^1$ and is independent from $A_n$, where $A_n$ is non-negative, compactly supported and $E(A_n^2) = 1$.
}\end{remark}
%%%%%%%%%%%%%%%%%%%%%%%%%%%%%%%%%%%%%%%%%%%%%%
\section{Probabilistic properties}
In this section $u_0$ is given by \eqref{data}. 
We now want to prove that until time of order $\varepsilon^{-1}$, the maps $a$, $b$ and $c$ are of order $0$ in $\varepsilon$.
For that, we use the following proposition : 
\begin{proposition} \label{kin}
There exist $C,c$ two positive constants such that for all $R>0$, the probability for the initial datum to have a $H^s$ norm bigger than $R$ satisfies : 
$$
P(\|u_0\|_{H^s} \geq R ) \leq C e^{-cR^2}\; .
$$ 
\end{proposition}
\begin{proof}
We first observe that \eqref{sub-gauss} together with the zero mean value assumption imply that
\begin{equation}\label{sub-gauss-bis}
E(e^{\gamma \textrm{Re} (g_n)})\leq e^{c\gamma^2},\quad E(e^{\gamma \textrm{Im} (g_n)})\leq  e^{c\gamma^2}\,  .
\end{equation}
First, we notice that thanks to \eqref{sub-gauss},  we only need to get \eqref{sub-gauss-bis} for small value of $|\gamma|$, say $|\gamma|\leq 1$
Next, we apply \eqref{sub-gauss} with $\gamma=\pm \alpha$ to get
$$
E(e^{\alpha |\textrm{Re} (g_n)|})+E(e^{\alpha |\textrm{Im} (g_n)|})<\infty.
$$
Now, we use that there exist two positive constants $C_1$ and $C_2$ such that for every $|\gamma|\leq 1$ and every $x\in\R$,
$$
|e^{\gamma x}-1-\gamma x|\leq C_1\gamma^2e^{C_2|x|}\,.
$$
Thanks to  the zero mean value assumption on $g_n$, the above analysis implies that there exists a constant $A$ such that
$$
E(e^{\gamma \textrm{Re} (g_n)})\leq 1+A\gamma^2\leq e^{c\gamma^2},
 $$
provided $c\geq A$. A similar argument applies for the imaginary part of $g_n$.
 Thus, we indeed have \eqref{sub-gauss-bis} and we are in a position to apply \cite[Lemma~3.1]{BT1}.
  \\
 
By separating the real and the imaginary parts, using \cite[Lemma~3.1]{BT1}, we obtain that there exist two positive constants $C$ and $c$ such that 
for every $y\geq 0$ and every sequence $(a_n)$,
$$
P( |\sum_n a_n g_n |\geq y) \leq C e^{-c y^2/ (\sum_n |a_n|^2)}
$$
Indeed, if $a_n = \alpha_n + i \beta_n$ and $g_n = h_n+il_n$, then, 
$$
|\sum_n a_n g_n | \leq |\sum_n \alpha_n h_n| + |\sum_n \alpha_n l_n| +|\sum_n \beta_n h_n|+|\sum_n \beta_n l_n|
$$
and therefore,
\begin{eqnarray*}
P( |\sum_n a_n g_n |\geq y) & \leq & P(|\sum_n \alpha_n h_n|\geq y/4) + P(|\sum_n \alpha_n l_n|\geq y/4) 
\\
& &
+P(|\sum_n \beta_n h_n|\geq y/4)+P(|\sum_n \beta_n l_n|\geq y/4)
\end{eqnarray*}
and then we can apply the \cite[Lemma~3.1]{BT1} on each term of the right hand side. Remark that since the $g_n$ are independent from each other, so are the $h_n$ and the $l_n$, even though $h_n$ is not necessarily independent from $l_n$.

We deduce from that that the $L^q$ norm (in the probability space ) of $\sum a_n g_n$ satisfies : 
$$
\| \sum_n a_n g_n \|_{L^q}^q \leq \Big( C q \sum_n |a_n|^2 \Big)^{q/2} 
$$
with $C$ independent from $a_n$ and $q$.
Indeed, this property is due to a change of variable and an induction on $q$. First, we have that : 
$$
\| \sum_n a_n g_n \|_{L^q}^q = \int q y^{q-1} P(|\sum a_n g_n |\geq y ) dy \leq \int C q y^{q-1}e^{-c y^2/\sum |a_n|^2}dy
$$
With $z = \frac{y}{\sqrt{\sum |a_n|^2}}$,
$$
\| \sum_n a_n g_n \|_{L^q}^q \leq \left( \sum |a_n|^2\right)^{q/2} C(q)
$$
with
$$
C(q) = C \int q z^{q-1}e^{-c z^2}dz\; .
$$
By integration by parts, we get : 
$$
C(q+2) = \frac{q+2}{2c} C(q)
$$
and then using that $C(q)$ is bounded uniformly in $q$ for $q\in [1,3]$, we get
$$
C(q) \leq C \left( \frac{q}{2c} \right)^{q/2}\leq (C q )^{q/2}
$$
and consequently
\begin{equation}\label{otz}
\| \sum_n a_n g_n \|_{L^q}^q \leq \left(Cq \sum |a_n|^2\right)^{q/2}
\end{equation}
Then, we use that : 
$$
P(\|u_0\|_{H^s} \geq R )  = P (\|u_0\|_{H^s}^q \geq R^q) \leq R^{-q} E( \|u_0\|_{H^s}^q) =  R^{-q} \|u_0\|_{L^q_P,H^s_x}^q
$$
where $L^q_P$ denotes the $L^q$ norm in the probability space and $H^s_x$ the $H^s$ norm in the physical space.

For $q\geq 2$ and thanks to Minkowski inequality, 
$$
 \|u_0\|_{L^q_P,H^s_x} =  \|\sum |n|^s\lambda_n g_n e^{inx}\|_{L^q_P,L^2_x}\leq  \|\sum |n|^s\lambda_n g_n e^{inx}\|_{L^2_x, L^q_P}
$$
Hence, using \eqref{otz} and \eqref{Hs}, we get
$$
 \|u_0\|_{L^q_P,H^s_x} \leq \|\Big( C q \sum_n |\lambda_n|^2\; |n|^{2s}\Big)^{1/2}\|_{L^2_x}
 =
 C\sqrt{q} \Big( \sum_n |\lambda_n|^2\; |n|^{2s}\Big)^{1/2}\leq C\sqrt{q}
$$
This in turn implies the bound
$$
P(\|u_0\|_{H^s} \geq R )  \leq \Big(\frac{Cq}{R^2}\Big)^{q/2} \; .
$$
Set $q(R) = e^{-1} \frac{R^2}{C}$ 
such that what inside the parenthesis in the above expression is equal to $e^{-1}$ in the particular case $q=q(R)$.
If $R$ is such that $q(R)\geq 2$ then we have that : 
$$
P(\|u_0\|_{H^s} \geq R )  \leq e^{-q(R)/2} = e^{-cR^2}\; .
$$
Let $R_0$ be defined by $2 = e^{-1} \frac{R_0^2}{C}$, i.e. $R_0=\sqrt{2eC}$. 
For $R\in [0,R_0]$, we can simply write
 $$
 P(\|u_0\|_{H^s} \geq R )  e^{cR^2}\leq e^{cR_0^2}
 $$
Therefore
$$
P(\|u_0\|_{H^s} \geq R )  \leq  e^{cR_0^2}\,e^{-cR^2}\; ,\quad \forall\, R\geq 0.
$$
This completes the proof of Proposition~\ref{kin}.
\end{proof}
As a consequence of  Proposition~\ref{kin}, we get the following exponential integrability statement.
\begin{lemma}\label{probbounds} 
There exists $\delta_0>0$ such that 
$$
E(e^{\delta_0 \|u_0\|_{H^s}^2})<\infty\,.
$$
\end{lemma}
We deduce from Lemma~\ref{energy} and Lemma~\ref{probbounds} , 
that as long as $|t|$ is bounded by $\varepsilon^{-1}$ the norm of $c$ can be bounded in probability. Indeed,
\begin{proposition} 
For all $p\geq 1$ there exists $C_p$ such that for all $|t| \leq \frac{1}{C_p \varepsilon}$, we have the bounds : 
$$
E(\|c(t)\|_{L^2}^p)^{1/p} \leq C_p t^2
$$
for KP-I,
$$
E(\|c(t)\|_{L^2}^p)^{1/p} \leq C_p |t|
$$
for KP-II,
$$
E(\|c(t)\|_{H^1}^p)^{1/p}\leq C_p |t|
$$
for BBM.
\end{proposition}
\begin{proof}
This comes from the fact that 
$$
E(\|u_0\|_{H^s}^p)
$$
is bounded for all $p$ and that
$$
E(e^{c\varepsilon |t|\|u_0\|_{H^s}})\; , \; E(e^{c\varepsilon^2|t|\|u_0\|_{H^s}^2})\; ,\; E(e^{c\varepsilon^2t^2\|u_0\|_{H^s}^2})
$$
are bounded uniformly in $t$ as long as $t^2c\varepsilon^2$ is less than the $\delta_0$ defined in Lemma \ref{probbounds}.
\end{proof}
We next collect some properties of the random variables $(g_n)$.
\begin{lemma}\label{case}
Under our assumption on $(g_n)$, with the $n_j$ belonging to $\D^d$,
$
E(g_{n_1}g_{n_2})=\delta_{n_1}^{-n_2}$ and $E(g_{n_1}g_{n_2}g_{n_3})=0.
$
Moreover
$
E(g_{n_1}g_{n_2}g_{n_3}g_{n_4})=0,
$
unless $n_1=-n_{j}$ for some $j\in\{2,3,4\}$ and $n_k=-n_l$ for the two indexes $k,l$ in the set $\{1,2,3,4\}/\{1,j\}$.
Moreover $E(g_{n_1}g_{n_2}g_{n_3}g_{n_4})=1$ if $n_1\neq n_k$ and $n_1\neq -n_k$. Finally  
$
E(g_{n_1}g_{n_2}g_{n_3}g_{n_4})=E(|g_{n_1}|^4)
$ if $n_1= n_k$ or $n_1=-n_k$.
\end{lemma}
The proof of this lemma follows by using the independence assumption via a 
careful case by case study. In particular, we use that under our assumption of symmetry of the distribution we have that
$$
E(g_n^3)=E(|g_n|^2g_n) =0\; , \; E(g_n^4)=E(|g_n|^2g_n^2) = E(|g_n|^2\overline{g_n}^2) = E(\overline{g_n}^4) = 0\; .
$$
%%%%%%%%%%%%%%%%%%%%%%%%%%%%%%%%%%%%%%%%%%%%%%%%%
%%%%%%%%%%%%%%%%%%%%%%%%%%%%%%%%%%%%%%%%%%%%%%%%%%
\section{Expansion of the covariances}
In this section, we complete the proof of Theorem~\ref{result}.
Using the fact that $a$, $b$, and $c$ are of order $0$ in $\varepsilon$ as long as $t \lesssim \varepsilon^{-1}$ 
we would like to develop the covariances of the amplitudes of the different wavelengths.
Let $d_n^m(t)$ be defined as
$$
d_n^m(t) = E(\overline{v_m}(t) v_n(t) )\; .
$$
Then we have the following statement.
%%%%%%%%
\begin{proposition}\label{posledno}
We have that
$$
\partial_t d_n^m(t) = \delta_n^m\varepsilon^2 \partial_t G_n(\lambda,t) + \varepsilon^3 r(\varepsilon;t,m,n),
$$
where $t\,r(\varepsilon;t,m,n)$ satisfies the bounds for $R(\varepsilon;t,m,n)$ announced in the statement of Theorem~\ref{result}.
\end{proposition}
%%%%%
\begin{proof}
Let us compute the time derivative of $d_n^m(t)$ .
Since $v_n$ satisfies
\begin{equation*}
\partial_t v_n(t) = -i\varepsilon \varphi(n) \sum_{k+l=n} v_k v_l e^{i\Delta_n^{k,l}t}\; ,
\end{equation*}
it comes
\begin{equation}\label{expand}
\partial_t d_n^m (t) = i\varepsilon \varphi(m) 
E\Big( \sum_{k+l= m}\overline v_k \overline v_l v_n e^{-i\Delta_m^{k,l}t}\Big) -
i\varepsilon\varphi(n) E\Big( \sum_{k+l=n} v_k v_l \overline v_m e^{i\Delta_n^{k,l}t} \Big)\; .
\end{equation}
In the cases of the KP equations, the term of last order, that is the term of order $7$ in $\varepsilon$ will involve three occurrences of $c$, 
and since we only have a bound for $c$  in $L^2$, we will not be able to bound 
$$
\varphi(n)E\Big( \sum_{k+l=n} c_k c_l \overline c_m e^{i\Delta_n^{k,l}t} \Big)
$$
by some function depending on the time and not on $n,m$ (see Remark~\ref{rem})
\\

By inserting 
$
v_n=v_n(\varepsilon) = a_n + \varepsilon b_n + \varepsilon^2 c_n(\varepsilon)
$
in \eqref{expand}
we distinguish different cases according to the power of $\varepsilon$.

First, it is clear that the term of order $0$ in the expression of $\partial_t d_n^m$ is $0$.

Then the term of order $1$ is also $0$ since it involves three occurrences of $a$ : $a_k a_l a_n$, and $a_k = \lambda_k g_k$.
Thus, we can apply Lemma~\ref{case}. This cancellation is frequently used in the Physics literature on the subject.

We will describe the term of order 2 later.

The term of order 3 involves combinations of 1 $c$ and 2 $a$ or 2 $b$ and 1 $a$. Hence, in the KP-I case it is less than $C\max(|n|,|m|) t^2$. 
In the case of KP-II, because of the different estimate on $b$, it is less than $C\max(|n|,|m|) |t|$. 
A similar analysis applies in the BBM case to get the bound $C(\min (|m|,|n|)^{-1}|t|$.

Let us describe this bound in the particular case of combinations between 1 $c$ and 2 $a$, the other ones resulting from similar computations. If we replace one occurrence of $v$ by $c$ and two occurrences of $v$ by $a$ in the expression
$$
i \varphi(m) 
E\Big( \sum_{k+l= m}\overline v_k \overline v_l v_n e^{-i\Delta_m^{k,l}t}\Big) -
i\varphi(n) E\Big( \sum_{k+l=n} v_k v_l \overline v_m e^{i\Delta_n^{k,l}t} \Big)
$$
we get
\begin{eqnarray*}
i \varphi(m) 
E\Big( \sum_{k+l= m}\left(\overline a_k \overline a_l c_n + \overline a_k \overline c_l a_n+ \overline c_k \overline a_l a_n\right)e^{-i\Delta_m^{k,l}t}\Big) - \\
i\varphi(n) E\Big( \sum_{k+l=n} \left( a_k a_l \overline c_m +a_k c_l \overline a_m+c_k a_l \overline a_m\right) e^{i\Delta_n^{k,l}t} \Big)\; .
\end{eqnarray*}
For KP, we can bound the $L^2$ norm of $c$, hence, since by a Cauchy-Schwartz inequality on the $a$ : 
$$
\Big|\sum_{k+l= m}\overline a_k \overline a_l c_ne^{-i\Delta_m^{k,l}t}\Big| \leq |c_n| \; \|u_0\|_{L^2}^2 \leq \|c\|_{L^2}\|u_0\|_{L^2}^2
$$
and by taking its expectation and circular arguments for the other terms, we get the bound on this term: 
$$
3 \left( |\varphi(n)|+|\varphi(m)|\right)E(\|c\|_{L^2}^3)^{1/3}E(\|u_0\|_{H^1}^3)^{2/3} \; .
$$
For KP-I, $E(\|c\|_{L^2}^3)^{1/3}$ is bounded by $Ct^2$ and for KP-II it is bounded by $C|t|$. Hence, as $\varphi(n) = n_1$, this term is bounded by $C\max(|n|,|m|) t^2$ for KP-I and $C\max(|n|,|m|) |t|$ for KP-II.
For BBM, $E(\|c\|_{L^2}^3)^{1/3}\leq E(\|c\|_{H^1}^3)^{1/3}$ is bounded by $C(1+|t|)$. Hence, as $|\varphi(n)| \leq  |n|^{-1}$, this term is bounded by $C(\min (|m|,|n|)^{-1}|t|$.

The third order in $\varepsilon$ also involves combinations of 2 $b$ and 1 $a$. In this case, the order in time for KP-II and BBM is $0$. Hence, the term of third order is bounded by $C\max(|n|,|m|) t^2$ for KP-I, $C\max(|n|,|m|)(1+ |t|)$ for KP-II and $C\min(|n|,|m|)^{-1}(1+|t|)$ for BBM.

The term of order 4 involves combinations of 1 $c$, 1 $b$ and 1 $a$ or 3 $b$. Hence, in the KP-I case it is less than $C\max(|n|,|m|)|t|^3$. 
A similar analysis applies in the KP-II and BBM cases.

The term of order 5 involves combinations of 1 $a$ and 2 $c$ or 2 $b$ and 1 $c$. Hence, in the KP-I case it is less than $C\max(|n|,|m|)|t|^4$.
Again a similar analysis applies in the KP-II and BBM cases.

The term of order 6 involves combinations of 1 $b$ and 2 $c$. Hence, in the KP-I case it is less than $C\max(|n|,|m|)|t|^5$ and a similar analysis applies 
in the KP-II and BBM cases. 
 
Finally the term of order 7 involves combinations of 3 $c$. Hence, it is less than $C\max(|n|,|m|)|t|^6$ in the KP-I case and $C\max(|n|,|m|)|t|^3$ in
the KP-II case. 

Since $t$ is less than $\varepsilon^{-1}$, we have that all estimates in the KP-I case are $O( \max(|n|,|m|) \varepsilon^3 t^2)$.
In the KP-II case they are $O( \max(|n|,|m|) \varepsilon^3 (1+|t|))$ and in the BBM case $O( (\min(|m|,|n|)^{-1} \varepsilon^3(1+ |t|))$.

Let us compute the term of order 2. As it involves 2 $a$ and 1 $b$, two sums of different nature (and their complex conjugate when inverting $n$ and $m$) appear in it : 
$$
V_n^m (t) = i \varphi(m) E(\Big( \sum_{k+l= m}\overline a_k \overline a_l b_n e^{-i\Delta_m^{k,l}t}\Big)
$$
and
$$W_n^m(t) = i \varphi(m) E(\Big( \sum_{k+l= m}\overline b_k \overline a_l a_n e^{-i\Delta_m^{k,l}t}\Big)
$$
which appears twice because of the symmetry between $k$ and $l$. The term of order 2 is therefore equal to
$$
V_n^m(t)+\overline V_m^n(t) +2( W_n^m(t)+\overline W_m^n(t))\; .
$$

By replacing $b_n(t)$ by its value 
$$b_n (t) = -i\varphi(n) \sum_{j+q = n} a_ja_q F_n^{j,q}(t)$$
we get
$$V_n^m(t) = \varphi(n) \varphi(m) \sum_{k+l=m}\sum_{j+q=n} E(\overline a_k \overline a_l a_j a_q) e^{-i\Delta_m^{k,l}t}F_n^{j,q}(t)\; .$$
We now recall that thanks to Lemma~\ref{case}, $E(\overline a_k \overline a_l a_j a_q)$ is equal to zero unless we can pair the indexes.
We can not pair $k$ with $l$ or we will have $k=l$, $m=0$, and $m\neq 0$ since $m_1\neq 0$ by hypothesis, but we can pair $k$ with $j$ and $l$ with $q$ or $k$ with $q$ and $l$ with $j$. In both case we have $n=m$. As long as $k\neq l$, we have : 
$$
E(\overline a_k \overline a_l a_j a_q) = |\lambda_k|^2 |\lambda_l|^2
$$
otherwise $2k=n$ and
$$
E(\overline a_k \overline a_l a_j a_q)= \delta_n^{2k}|\lambda_k|^4 E(|g_k|^4)\; .$$

Using our assumptions on the random variables and that $
e^{-i\Delta_n^{k,l}t}F_n^{k,l}(t)=-F_n^{k,l}(-t)
$
we get 
$$V_n^m(t)  = 2\delta_n^m \varphi^2(n) \sum_{k+l=n} |\lambda_k|^2 |\lambda_l|^2 (-F_n^{k,l}(-t))
+(E(|g_n|^4)-2)\delta_{n}^m\varphi^2(n)\delta_{n}^{2q}(-F^{q,q}_{n}(-t))|\lambda_q|^4\; .$$
Next,
$$W_n^m (t)= - \varphi(m) \sum_{k+l=m}\sum_{j+q=k} \varphi(k) E(\overline a_j \overline a_q \overline a_l a_n) e^{-i\Delta_m^{k,l}t} \overline {F_k^{j,q}(t)} \; .$$
Here, we can pair $j$ with $l$ and $q$ with $n$ or $j$ with $n$ and $q$ with $l$ but not $j$ with $q$. 
In both case, we can do the computation (with changing the indexes): 
$$
n = q = k-j = k+l= m
$$ 
As long as $l\neq n$, we get: 
$$
E(\overline a_k \overline a_l a_j a_q)= |\lambda_n |^2 |\lambda_l|^2
$$
but otherwise 
$$
E(\overline a_k \overline a_l a_j a_q) = |\lambda_n|^4 E(|g_n|^4)\; .
$$
Again, using our assumptions on the random variables, we get
\begin{eqnarray*}
W_n^m (t) & = & -2\delta_n^m  \varphi(n) \sum_{k+l=n}  \varphi(k) |\lambda_n|^2 |\lambda_l|^2 e^{-i\Delta_n^{k,l}t} \overline {F_k^{n,-l}(t)} 
\\
& &
+
(E(|g_n|^4)-2)\delta_{n}^m\varphi(2n)\varphi(n)
e^{-i\Delta_{n}^{2n,-n}t} 
\overline{F^{n,n}_{2n}(t)}
|\lambda_n|^4
\end{eqnarray*}
and since
$$
F_k^{n,-l}(t)=\overline {F_n^{k,l}(t)},
$$
we arrive at
\begin{eqnarray*}
W_n^m(t) & = & -2\delta_n^m \varphi(n) \sum_{k+l=n}  \varphi(k) |\lambda_n|^2 |\lambda_l|^2 e^{-i\Delta_n^{k,l}t} (-F_n^{k,l}(-t))
\\
& &
+(E(|g_n|^4)-2)\delta_{n}^m\varphi(2n)\varphi(n)
(-F^{2n,-n}_{n}(-t))|\lambda_n|^4\,\,.
\end{eqnarray*}
Combining the previous formulae implies the claimed expression for the second order.
This completes the proof of Proposition~\ref{posledno}.
\end{proof}
%%%%%%%%
Observe that
$$
d_n^m(t)=
e^{-it(\omega(n)-\omega(m))}\, E(\overline{u_m}(t) u_n(t) )
$$
and therefore
$$
E(\overline{u_m}(t) u_n(t) )-E(\overline{u_m}(0) u_n(0) )
=
e^{it(\omega(n)-\omega(m))}d^m_n(t)-d^m_n(0).
$$
If $m=n$, it suffices to employ the fundamental theorem of calculus to the function $d_n^m$ and to apply Proposition~\ref{posledno}.
If $m\neq n$ then one has that $d^m_n(0)=0$ and hence one may write 
$$
e^{it(\omega(n)-\omega(m))}d^m_n(t)-d^m_n(0)
=
e^{it(\omega(n)-\omega(m))}(d^m_n(t)-d^m_n(0))
$$
and apply again the fundamental theorem of calculus in combination with Proposition~\ref{posledno}.
\\

%%%%%%%%%%%%%%%%%%%%%%%%%%%%%
%%%%%%%%%%%%%%%%%%%%%%%%%%%%%%
Let us now bound $G_n(\lambda,t)$.
For KP-I, we use the fact that $\varphi(n) = n_1$ and $|F_n^{k,l}(t)| \leq |t|$. Then, as the term involving $E(|g_n|^4)$ is included in the sum,
$$
|\partial_t G_n(t)| \leq C |t| 
\Big( |n|^2 \sum_{k+l=n} |\lambda_k|^2 |\lambda_l|^2 + |n| |\lambda_n|^2 \sum_{k+l = n}|k|\; |\lambda_l|^2\Big),
$$
$$
\sum_{k+l = n}|k|\; |\lambda_l|^2 \leq (|n| \sum_l |\lambda_l|^2 + \sum_l |l| |\lambda_l|^2 ) \leq C|n| \; \|u_0\|_{H^s}^2\leq C |n|,
$$
$$
|n|^{2s}\sum_{k+l=n} |\lambda_k|^2 |\lambda_l|^2 \leq 2 \sum_{k+l=n} |k|^{2s}|\lambda_k|^2 |\lambda_l|^2 \leq 2 \max (|\lambda_l|^2) \|u_0\|^2_{H^s}\leq C
$$
and since $|\lambda_n|^2 |n|^{2s}$ is bounded,
$$
|\partial_t G_n(t)| \leq C |t|\; |n|^{2-2s},\quad| G_n(t)| \leq C |t|^2 |n|^{2-2s}\; .
$$

For KP-II and BBM, we start by integrating $\partial_t G_n$ before bounding it. We recall that initially $G_n$ is null and that $F_n^{k,l}(t) = \int_{0}^t e^{i\Delta_n^{k,l}\tau} d\tau$. Therefore, we can write 
\begin{eqnarray*}
G_n(t)&  = & 4\varphi(n)  \sum_{k+l=n} \widetilde F_n^{k,l}(t) \left( \varphi(n) |\lambda_k|^2 |\lambda_l|^2 - \varphi(k) |\lambda_l|^2 |\lambda_n|^2 - \varphi(l) |\lambda_n|^2 |\lambda_k|^2\right) \\
 & & + (E(|g_n|^4) - 2) \varphi(n)\left( 2\delta_{2q}^n \widetilde F_n^{q,q}(t) \varphi(n) |\lambda_q|^4 - 4 \widetilde F_n^{2n,n}(t) \varphi(2n) |\lambda_n|^4 \right)
\end{eqnarray*}
with 
$$
\widetilde F_n^{k,l} (t) = -\int_0^t \textrm{Re}(F_n^{k,l}(-\tau)) d\tau = \frac{1- \cos(\Delta_n^{k,l}t)}{(\Delta_n^{k,l})^2} \; .
$$
We notice that $\varphi(n) |\lambda_q|^4$ is of the form $\varphi(n) |\lambda_k|^2 |\lambda_l|^2$ and that $\varphi(2n) |\lambda_n|^4$ is of the form $\varphi(l) |\lambda_n|^2 |\lambda_k|^2$ with $k+l=n$ to produce the bound
$$
|G_n(t)| \lesssim |\varphi(n)| \sum_{k+l=n} |\widetilde F_n^{k,l}(t)|\left( | \varphi(n)|\; |\lambda_k|^2 |\lambda_l|^2 + |\varphi(k)|\; |\lambda_l|^2 |\lambda_n|^2 +| \varphi(l)|\; |\lambda_n|^2 |\lambda_k|^2\right) \; ,
$$
and then we get a uniform in time bound using the form of $\widetilde F_n^{k,l}(t)$
$$
|G_n(t)| \lesssim |\varphi(n)| \sum_{k+l=n} (\Delta_n^{k,l})^{-2 } \left( | \varphi(n)|\; |\lambda_k|^2 |\lambda_l|^2 + |\varphi(k)|\; |\lambda_l|^2 |\lambda_n|^2 +| \varphi(l)|\; |\lambda_n|^2 |\lambda_k|^2\right) \; .
$$

We now focus on the case of KP-II. For this equation, we have seen that $(\Delta_n^{k,l})^{-2}$ is bounded by $(k_1l_1n_1)^{-2}$. We recall that $\varphi(n) = n_1$. Hence, we get
$$
|G_n(t)| \lesssim |n_1| \sum_{k+l=n} (n_1k_1l_1)^{-2 } \left( | n_1|\; |\lambda_k|^2 |\lambda_l|^2 + |k_1|\; |\lambda_l|^2 |\lambda_n|^2 +| l_1|\; |\lambda_n|^2 |\lambda_k|^2\right) \; .
$$
Using symmetries on $k$ and $l$, we bound $G_n(t)$ by two different sums $G_n(t) \lesssim I+II$ with
$$
I = |n_1| \sum_{k+l=n} (n_1k_1l_1)^{-2 }  | n_1|\; |\lambda_k|^2 |\lambda_l|^2 \mbox{ and } II = |n_1| \sum_{k+l=n} (n_1k_1l_1)^{-2 }|k_1|\; |\lambda_l|^2 |\lambda_n|^2 \; .
$$
We multiply $I$ by $|n|^{2s}$, use symmetries and that $|n|^{2s} \leq C (|k|^{2s}+|l|^{2s})$ to get
$$
|n|^{2s} I \lesssim \sum_{k+l=n}  |k|^{2s} |\lambda_k|^2 |\lambda_l|^2 \; .
$$
We get that the sum is bounded uniformly in $n$ since $\sum_k |k|^{2s}|\lambda_k|^2$ is bounded by hypothesis and thus so is $\sup_l |\lambda_l|^2$.

Multiplying $II$ by $|n|^{2s}$ gives
$$
|n|^{2s} II = |n|^{2s} |\lambda_n|^2 \sum_l |n_1|^{-1} |n_1-l_1|^{-1} l_1^{-2} |\lambda_l|^2
$$
and since $|n|^{2s} |\lambda_n|^2$ is bounded, $|n_1|^{-1} |n_1-l_1|^{-1} l_1^{-2} \leq 1$ and $\sum_l |\lambda_l|^2$ is bounded, we get the result. Summing $I$ and $II$ gives the estimate
$$
|G_n(t)| \leq C |n|^{-2s}
$$
where the constant $C$ depends on the initial datum (that is the law of $g_n$ and the sequence $(\lambda_n)_n$) but not on time.

We now focus on the case of BBM. For this equation, we have seen that $\Delta_n^{k,l}$ is equal to $-\frac{kln(3+n^2 - kl)}{(1+n^2)(1+k^2)(1+l^2)}$ and thus $(\Delta_n^{k,l})^{-2}$ is less than $C \frac{n^2 k^2 l^2}{k^4+l^4}$. We recall that $|\varphi (n)| = \frac{|n|}{1+n^2} \leq |n|^{-1}$. Hence, we get that
$$
|G_n(t)|\lesssim |n|^{-1} \sum_{k+l=n} \frac{n^2 k^2 l^2}{k^4+l^4} \left( |n|^{-1} |\lambda_k|^2 |\lambda_l|^2 + |k|^{-1} |\lambda_l|^2 |\lambda_n|^2 +|l|^{-1} |\lambda_n|^2 |\lambda_k|^2\right)\; .
$$
We proceed in the same way as KP-II by dividing the bound in to two sums $I$ and $II$ with 
$$
I =  \sum_{k+l=n} \frac{ |k|^{2-2s} |l|^{2-2s}}{k^4+l^4}   |k|^{2s}|\lambda_k|^2 |l|^{2s}|\lambda_l|^2
$$
and
$$
II = |n|^{2s} |\lambda_n|^2\sum_{k+l=n} \frac{|n|^{1-2s} |k|\; |l|^{2-2s}}{k^4+l^4} |l|^{2s} |\lambda_l|^2 \; .
$$
We multiply $I$ by $|n|^{\beta(s)}$, with $\beta(s) = 4s$ if $s\leq 1$ and $\beta(s) = 2+2s$ otherwise, use that $|n|^{\beta(s)} \leq C (|k|^{\beta(s)}+|l|^{\beta(s)})$ and symmetries to get
$$
|n|^{\beta(s)} I \lesssim \sum_{k+l=n}\frac{|k|^{\beta(s)+2-2s}|l|^{2-2s}}{k^4+l^4} |k|^{2s}|\lambda_k|^2 |l|^{2s}|\lambda_l|^2 \; .
$$
If $s\geq 1$ then $\beta(s) = 2+2s$ and $2-2s \leq 0$ thus $\frac{|k|^{\beta(s)+2-2s}|l|^{2-2s}}{k^4+l^4} \leq \frac{k^4}{k^4+l^4} \leq 1$ and we use the conditions on the initial datum to bound $|n|^{\beta(s)} I$ independently from $n$ and $t$. If $s\leq 1$, $\beta(s) = 4s$ thus $|k|^{\beta(s)+2-2s}|l|^{2-2s} = |k|^{2+2s}|l|^{2-2s} \leq k^4 + l^4$ and thus we get the bound.

Multiplying $II$ by $|n|^{\beta(s)}$ gives
$$
|n|^{\beta(s)} II = |n|^{2s} |\lambda_n|^2\sum_{k+l=n} \frac{n^{1-2s+\beta(s)} |k|\; |l|^{2-2s}}{k^4+l^4} |l|^{2s} |\lambda_l|^2 \; .
$$
Since $\beta(s)+1-2s \geq 0$ we can write $|n|^{\beta(s)} II \lesssim |n|^{2s} |\lambda_n|^2(II.1+II.2)$ with
$$
II.1 = \sum_{k+l=n} \frac{ |k|^{2-2s+\beta(s)} |l|^{2-2s}}{k^4+l^4} |l|^{2s} |\lambda_l|^2 \mbox{ and } II.2 = \sum_{k+l=n} \frac{|k|\; |l|^{3-4s+\beta(s)}}{k^4+l^4} |l|^{2s} |\lambda_l|^2\; .
$$
We have already seen that $|n|^{2s} |\lambda_n|^2$ and $\frac{ |k|^{2-2s+\beta(s)} |l|^{2-2s}}{k^4+l^4}$ are bounded independently from $n$. Then, if $s\leq 1$, $\beta(s) = 4s$ thus $|k|\; |l|^{3-4s+\beta(s)} = |k|\; |l|^3 \leq k^4+l^4$ and otherwise $\beta(s) = 2+2s$ thus $|k|\; |l|^{3-4s+\beta(s)} = |k|\; |l|^{5-2s} \leq |k|\; |l|^3$, which gives the bound. Summing $I$ and $II$ gives the estimate
$$
|G_n(t)| \leq C |n|^{-\beta(s)}
$$
where the constant $C$ depends on the initial datum but not on time.

This completes the proof of Theorem~\ref{result}.
%%%%%%%
\begin{remark}{\rm
In the case of the KP equations,
formally, if $\lambda_k$ do not depend on $k$  and $g_n$ are standard complex Gaussian variables then
$\partial_t G_n(\lambda, t)=0$. This goes with the fact that the measure induced by
$$\sum_n g_n e^{in x} $$
should be formally invariant through the flow of KP, as it is a renormalization of the formal object  
$$e^{- ||u||_{L^2}^2}du$$
and the $L^2$ norm of the solution of KP does not depend on time. 
However the support of these measures in the case of the KP equation contains functions which are too singular for the available well-posedness theory.
In the case of the KdV equation the measures is supported by $H^{-1/2-}$.  This could be a motivation to try to lower the regularity assumption
in our approach in the KdV case. }
\end{remark}

\end{document}